\documentclass[12pt,reqno]{amsart}
\usepackage{amsfonts}
\usepackage{multirow}
\usepackage{a4wide}
\usepackage{longtable}
\usepackage{supertabular}

\allowdisplaybreaks
\numberwithin{equation}{section}

\newtheorem{theorem}{Theorem}[section]

\newtheorem{lemma}[theorem]{Lemma}

\theoremstyle{definition}

\newtheorem{remark}[theorem]{Remark}

\newcommand{\R}{\mathbb{R}}
\newcommand{\ds}{\displaystyle}

\begin{document}

\title
[Proportional  positive solutions and least energy solutions ] { Existence, non-degeneracy of proportional  positive solutions and least energy solutions  for a fractional  elliptic system}

 \author{QiHan  He,\,\, Shuangjie Peng \,\,and \,\,Yan-Fang Peng }

\address{Department of Mathematics and Information Science, Guangxi University, Nanning, 530003, P. R. China }

\email{ heqihan277@163.com }

\address{School of Mathematics and Statistics, Central China
Normal University, Wuhan, 430079, P. R. China}

\email{ sjpeng@mail.ccnu.edu.cn}

\address{School of Mathematics, Guizhou Normal University, Guiyang, 550001, P. R. China}
\email{ pyfang2005@sina.com }

\maketitle
\begin{abstract}
In this paper, we study the following fractional  nonlinear Schr\"{o}dinger system
$$
\left\{%
\begin{array}{ll}
(-\Delta)^s u +u=\mu_1 |u|^{2p-2}u+\beta |v|^p|u|^{p-2}u,~~x\in \R^N,\vspace{2mm}\\
(-\Delta)^s v +v=\mu_2 |v|^{2p-2}v+\beta |u|^p|v|^{p-2}v,~~x\in \R^N,
\end{array}%
\right.
$$
where  $0<s<1, \mu_1 >0, \mu_2>0, 1<p<2_s^*/2, 2_s^*=+\infty$ for $N\le 2s$ and $2_s^*=2N/(N-2s)$ for $N>2s$, and $\beta \in \R$ is a coupling constant. We investigate
the existence and non-degeneracy of proportional positive vector solutions  for the above system  in some ranges of  $\mu_1,\mu_2, p, \beta$.
We also prove that the least energy vector solutions  must be  proportional and unique under some additional assumptions.
\end{abstract}
\section{Introduction}

In this paper,  we consider the following fractional   Schr\"{o}dinger system
\begin{equation}\label{1.1}
\left\{%
\begin{array}{lll}
(-\Delta)^s u +u=\mu_1 |u|^{2p-2}u+\beta |v|^p|u|^{p-2}u,~~x\in \R^N,\vspace{2mm}\\
(-\Delta)^s v +v=\mu_2 |v|^{2p-2}v+\beta |u|^p|v|^{p-2}v,~~x\in \R^N,\vspace{2mm}\\
u,v\in H^s(\R^N),
\end{array}%
\right.
\end{equation}
where $0<s<1, \mu_1 >0, \mu_2>0, 1<p<2_s^*/2,$ $2_s^*=+\infty$ for $ N \leq 2s$ and $2_s^*=2N/(N-2s)$ for $ N>2s$, $\beta \in \R$ and
$$
H^s(\R^N):=\Big\{u\in L^2(\R^N):\ds\int_{\R^N}(1+|\xi|^{2s})|\hat{u}|^2~d\xi<+\infty\Big\},
$$
where  $\hat{}$  denotes the
Fourier transform, i.e. $\hat{\phi}(\xi)=\frac{1}{(2\pi)^{N/2}}\int_{\R^N}e^{-2\pi i\xi\cdot x}\phi(x)dx$, $(-\Delta)^s$  of a function $\phi\in S$ is defined by
$$
\widehat{(-\Delta)^s\phi(\xi)} = |\xi|^{2s}\hat{\phi}(\xi).
$$
Here $S$ denotes the Schwartz space of rapidly decreasing $C^\infty$ functions in $\R^N$. In fact, since $S$ is density in $L^2(\R^N)$, $(-\Delta)^s$  can act on $H^s(\R^N).$  If $\phi$ is smooth enough, it can be expressed
by the following formula
$$
(-\Delta)^s\phi(x)=C_{N,s} P.V.\int_{\R^N}\frac{\phi(x)-\phi(y)}{|x-y|^{N+2s}}dy=\ds C_{N,s} \lim_{\varepsilon\rightarrow 0}\int_{\R^N\setminus B_\varepsilon(x)}\frac{\phi(x)-\phi(y)}{|x-y|^{N+2s}}dy,
$$
where $P.V.$ is the principal value and $C_{N,s}$ is a normalization constant.

This type of  fractional  Schr\"{o}dinger systems are of
particular interest in fractional quantum mechanics for the study of
particles on stochastic fields modelled by L\'{e}vy processes. A
path integral over the L\'{e}vy flights paths and a fractional
Schr\"{o}dinger equation of fractional quantum mechanics are
formulated by Laskin \cite{l1} from the idea of Feynman and Hibbs's
path integrals (see also \cite{l2}).

Problem \eqref{1.1} can be regarded  as a counterpart of the following fractional  equation
 \begin{equation}\label{1.3}
 (-\Delta)^su+u=|u|^{2p-2}u,~~~x\in \R^N.
\end{equation}
When $s=1$, \eqref{1.3} turns to be the classical equation
\begin{equation}\label{1.4}
-\Delta u+u=|u|^{2p-2}u,
\end{equation}
where $1<p<2_1^*/2$.  In \cite{c1}, Coffman showed the   uniqueness of ground state solutions of the following equation
$$
-\Delta u+u-u^3=0,\,x\in \R^3.
$$
For a general case,    the uniqueness of  positive radial solutions of
$$
\Delta u+f(u)=0,\,x\in \R^N,
$$
was obtained by Maris in \cite{m}  when  $N>1$ and $f(u)$ satisfies certain  assumptions.
In a celebrated paper \cite{k}, Kwong established
the uniqueness and non-degeneracy of the ground states for problem \eqref{1.4} for $ N\geq1$, which provides
an indispensable basis for the blow-up analysis as well as the stability of solitary waves for
related time-dependent equations such as the nonlinear Schr\"{o}dinger equation  (see
 \cite{mr}).

For \eqref{1.3} with $0<s<1$, the uniqueness of ground state solutions for the following  nonlinear model
$$
(-\Delta)^\frac{1}{2} u+u-u^2=0,\,x\in \R,
$$
was proved by  Amick and Toland \cite{at}.
In \cite{fls}, Frank, Lenzmann and Silvester  showed the  uniqueness and
non-degeneracy of ground state solutions $w$
for arbitrary space dimensions $N\geq 1$ and all admissible exponents $1<p<2^*_s/2$, where
the non-degeneracy means that the kernel of the associated linearized operator in $H^s(\R^N)$
$$
L_+=(-\Delta)^s+I-(2p-1)w^{2p-2}
$$
is exactly $\hbox{span}\{\frac{\partial w}{\partial x_i}:~i=1,2,\cdots,N\}$.
This result  generalizes
the uniqueness and  non-degeneracy  result for dimension $N = 1$  obtained in \cite{fl} and in particular, the uniqueness result in \cite{at}.
 The existence and symmetry results for the solution  $w$  for equation \eqref{1.3}  were also shown   by  Dipierro, Palatucci and Valdinoci
in \cite{dpv} and  Felmer, Quaas and Tan in \cite{fqt}. Recently, for a critical semi-linear nonlocal equation involving the fractional Laplacian,
 D\'{a}vila, del Pino and   Sire \cite{dds}
proved the non-degeneracy of the manifold consisting of positive solutions.

 Since the important  result  of \cite{fls}, people began to
focus on the generalized form of \eqref{1.3}.
Based on minimization on the Nehari manifold,  Secchi \cite{s} found   solutions for  the following  class of fractional nonlinear Schr\"{o}dinger equations
 \begin{equation}(-\Delta )^su +V(x)u=|u|^{2p-2}u.
 \end{equation}
 Felmer, Quaas and Tan \cite{fqt} studied  the existence of positive solutions
for the fractional nonlinear Schr\"{o}dinger equation
\begin{equation}
(-\Delta)^su+u=f(x,u) ~\hbox{in}~\R^N, u>0, \lim\limits_{|x|\rightarrow +\infty}u(x)=0,
\end{equation}
and  analyzed regularity, decay and symmetry properties of
these solutions. In   \cite{c}, Chang obtained the existence
of ground state solutions for the following   fractional Schr\"{o}dinger equation
$$
(-\Delta)^su + V(x)u = f(x,u), \,x\in \R^N,
$$
 by means of   variational methods,   where $f(x, u)$ is asymptotically linear in $u $ at
infinity. For more results concerning the fractional equations and the related problems, we can refer to
\cite{bl,bs,k,mr,w} and the references therein.

We emphasize  that   although there is  wide  study on  existence, uniqueness  and non-degeneracy
for single fractional  equation, to our knowledge,  there are few papers dealing with fractional  system, with the exception of  \cite{dp},
where Dipierro and Pinamonti studied  the symmetry properties of solutions of elliptic system
\begin{equation}\label{1.6}
\left\{%
\begin{array}{ll}
(-\Delta)^{s_1} u =F_1(u,v),~~x\in \R^N,\vspace{2mm}\\
(-\Delta)^{s_2} v =F_2 (u,v),~~x\in \R^N.
\end{array}%
\right.
\end{equation}
Here $F_1,F_2 \in C^{1,1}_{loc}(\R^2), s_1, s_2 \in (0,1)$. As far as  \eqref{1.6} is concerned, we can find many results  for the case $s_1=s_2=1$
(see \cite{a,bdw,clw,lw2,pw,p} and the references therein).

In the present paper, we will focus on the existence, non-degeneracy of proportional vector solutions for fractional system \eqref{1.1}, and will investigate the form and the uniqueness  of the least energy vector solutions of \eqref{1.1}. More precisely,   our first goal is to prove an existence and  non-degeneracy result for proportional   positive solutions of \eqref{1.1}, where non-degeneracy of a solution $(U,~V)$ for \eqref{1.1}  means that the kernel of the linearized operator of \eqref{1.1} at $(U,~V)$ is given by span$\{(\theta(\beta)\frac{\partial w}
{\partial x_j},\frac{\partial w}
{\partial x_j})~|~j=1,2, \cdots, N\}$ with $\theta(\beta) \neq 0$.
 Non-degeneracy is very important  because it enables one to construct   solutions for many problems, see \cite{cwy,den,dy,pw,wy} for example. Our  second goal is to show that the least energy solutions of \eqref{1.1} must be proportional and unique. A similar result for the case $s=1$ has been proved  by Chen and Zou \cite{cz}.

 Before we state our main results,
 we introduce some notations. We call $(u,v)$  a least energy solution of \eqref{1.1} if
 $u \not\equiv 0, v\not\equiv 0$ satisfy \eqref{1.1} and $(u,v)$ makes the value of the corresponding functional the smallest among all the solutions of \eqref{1.1}.
Throughout this paper,  we denote by $w$ the solution, found by Frank, Lenzmann and Silvester in \cite{fls}, for the equation \eqref{1.3}. Without loss of generality, we assume that $\mu_1>\mu_2$.
 Set
 \begin{equation}\label{S}
S:=\inf\limits_{u\in H^s(\R^N)\setminus\{0\}}\Big\{\ds\int_{\R^N}(1+|\xi|^{2s})|\hat{u}|^2:\,\ds\int_{\R^N}|u|^{2p}=1\Big\},
\end{equation}
\begin{equation}\label{Smu1mu2}
S_{\mu_1,\mu_2}:=\inf\limits_{(u, v)\in H^s(\R^N)^2\setminus\{0,0\}}\frac{\ds\int_{\R^N}(1+|\xi|^{2s})(|\hat{u}|^2+|\hat{v}|^2)}
{\Big(\ds\int_{\R^N}(\mu_1|u|^{2p}+2\beta|u|^p|v|^p+\mu_2|v|^{2p})\Big)^\frac{1}{p}}.
\end{equation}
We define the following functions which is important in the analysis of the uniqueness and non-degeneracy of the least energy solutions.
\begin{equation}\label{ftau}
f(\tau):=\frac{1+\tau^2}{(\mu_1+2\beta\tau^p+\mu_2\tau^{2p})^\frac{1}{p}},
\end{equation}
$$ H_1(t)=\frac{1}{p-1}t^2-\frac{p-2}{p}(p-1)
^{-\frac{p}{p-2}}t^\frac{2(p-1)}{p-2}.$$

We also write
\begin{equation}\label{1.7}
D=\left\{%
\begin{array}{ll}
(0,1),\quad\quad~ \hbox{if}~1<p<2 ~\hbox{and}~(p-1)\mu_1^\frac{p-2}{2(p-1)}\mu_2^\frac{p}{2(p-1)}>\beta>0,\vspace{2mm}\\
(1,+\infty),\quad\quad~ \hbox{if}~p>2~\hbox{and}~\beta>(p-1)\mu_1,\vspace{2mm}\\
(0,+\infty),\quad\quad~ \hbox{if}~\hbox{otherwise},
\end{array}%
\right.
\end{equation}

\begin{equation}\label{1.7'}
\tilde{D}=\left\{%
\begin{array}{ll}
(0,1),\quad\quad~ \hbox{if}~1<p<2 ~\hbox{and}~(p-1)\mu^\frac{p-2}{2(p-1)}>\tilde{\beta}>0,\vspace{2mm}\\
(1,+\infty),\quad\quad~ \hbox{if}~p>2~\hbox{and}~\tilde{\beta}>(p-1)\mu,\vspace{2mm}\\
(0,+\infty),\quad\quad~ \hbox{if}~\hbox{otherwise},
\end{array}%
\right.
\end{equation}

 and
\begin{equation}\label{1.8}\left\{%
\begin{array}{ll}
(A_1)~ 2<p<\frac{2^*_s}{2},~ 0<\beta\leq(p-1)\mu_1,\vspace{2mm}\\
(A_2)~2<p<\frac{2^*_s}{2},~\mu_1 \geq \frac{\mu_2}{2}(\frac{p}{p-1})^{p-1},~\beta >(p-1)\mu_1,\vspace{2mm}\\
(A_3)~2<p<\frac{2^*_s}{2},~\mu_1<\frac{\mu_2}{2}(\frac{p}{p-1})^{p-1}, ~(p-1)\mu_1\leq\beta\leq\beta_0~
\hbox{or}~\beta\geq\max\{\beta_1,(p-1)\mu_1\},\vspace{2mm}\\

(A_4)~1<p<\min\{2,\frac{2^*_s}{2}\},~ \beta\geq (p-1)\mu_1^\frac{p-2}{2(p-1)}\mu_2^\frac{p}{2(p-1)},\vspace{2mm}\\
(A_5)1<p<\min\{2,\frac{2^*_s}{2}\},~0<\mu_1<\frac{p\mu_2}{2-p},\vspace{2mm}\\
 \quad\quad0<\beta\leq \min \Big\{\frac{p\mu_2-\mu_1(2-p)}{2}, \,\,
2(p-1)(2-p)^\frac{2-p}{2(p-1)}(\frac{1}{p})^\frac{p}{2(p-1)}\mu_1^\frac{p}{2(p-1)}\mu_2^\frac{p-2}{2(p-1)}\Big\},\vspace{2mm}\\
(A_6)~1<p<\min\{2,\frac{2^*_s}{2}\},\max\Big\{\frac{p\mu_2-\mu_1(2-p)}{2},
~0\Big\}<\beta<(p-1)\mu_1^\frac{p-2}{2(p-1)}\mu_2^\frac{p}{2(p-1)},\vspace{2mm}\\
(A_7) ~p=2, ~\beta \in (0, \mu_2)\cup(\mu_1, +\infty),
\end{array}%
\right.
\end{equation}
where $ 0<\beta_0<\beta_1$ solve
$$ \frac{2(p-1)\mu_1}{p\mu_2}- H_1\Big(\frac{\beta}{\mu_2}\Big)=0.$$

Our first  result is on  the existence and  non-degeneracy  of positive proportional vector solutions.

\begin{theorem}\label{Th1}~ Suppose that $0<s<1, 1<p<2^*_s/2, \mu_1>\mu_2>0, \beta>0$ and one of $A_i$ ($i=1,2,\cdots,7$) defined in \eqref{1.8} holds. Then \eqref{1.1} has  a  positive solution
$(U,V):=(k_1w, \tau_0 k_1w) $ in $H^s(\R^N) \times H^s(\R^N)$ which is non-degenerate,
 where $\tau_0\in D$ satisfies $\mu_1+\beta\tau^p-\mu_2\tau^{2p-2}-\beta\tau^{p-2}=0$  and $k_1^{2p-2}=(\mu_1+\beta\tau_0^p)^{-1}$.
\end{theorem}

\begin{remark}
Indeed, we prove that any solution of the form $(k_1w, \tau_0 k_1w) $ in $H^s(\R^N) \times H^s(\R^N)$ is non-degenerate,
 where $\tau_0\in D$ satisfies $\mu_1+\beta\tau^p-\mu_2\tau^{2p-2}-\beta\tau^{p-2}=0$  and $k_1^{2p-2}=(\mu_1+\beta\tau_0^p)^{-1}$.
\end{remark}

\begin{remark}
When $p=2, \beta \in [\mu_2, \mu_1]$, \eqref{1.1} has no positive solutions. In fact, suppose, to the contrary,  that
$(u,v)$ is a positive solution of \eqref{1.1}. Multiplying the first equation
in \eqref{1.1} by $v$ and the second equation by $u$, and then integrating  on $\R^N$ and
substracting together, we obtain $0=\int_{\R^N}uv[(\mu_1-\beta)u^2+(\beta-\mu_2)v^2]~dx\neq 0$, which is a contradiction.
\end{remark}

Our second result is for the case $\beta<0$.

\begin{theorem}\label{Th1.4}
Suppose that $0<s<1, 1<p<2^*_s/2, \mu_1>\mu_2>0$. Then   there exists a decreasing  sequence $\{\beta_k\} \subset (-\sqrt{\mu_1\mu_2},0)$
such that  for $\beta \in (-\sqrt{\mu_1\mu_2},0)\setminus\{\beta_k\}$,
 \eqref{1.1} has a  positive non-degenerate solution
$(U,V):=(k_1w, \tau_0 k_1w)$ in $H^s(\R^N) \times H^s(\R^N)$,  where $\tau_0\in D$ satisfies $\mu_1+\beta\tau^p-\mu_2\tau^{2p-2}-\beta\tau^{p-2}=0, k_1^{2p-2}=(\mu_1+\beta\tau_0^p)^{-1}$.
\end{theorem}

The last result describes  the form and the uniqueness  of the least energy solutions of \eqref{1.1}.
\begin{theorem}\label{th1.6}
Assume that $(u_0, v_0)$ is a positive  least energy solution of \eqref{1.1}, and one of the following conditions holds:

$(1)~~\beta>\mu_1, p=2$,

$(2)~~\beta>0, 1<p<2$.

Then $(u_0, v_0)=(k_{min}w_{x_0},  k_{min}\tau_{min}w_{x_0})$. Moreover,  the least energy solution is unique.

Here $x_0$ is some point in $\R^N$, $w_{x_0}(x):=w(x-x_0),$  $\tau_{min}>0$ is the minimum point of $f(\tau)$ in $[0, +\infty)$  satisfying  $f(\tau_{min}):=\min\limits_{\tau \geq0}f(\tau)<f(0)=\mu_1^{-\frac{1}{p}}$
with  $k_{min}^{2p-2}=(\mu_1+\beta \tau_{min}^p)^{-1}.$
\end{theorem}

\begin{remark}
Each non-zero  local maximum or non-zero  minimum point of $f(\tau)$ corresponds to a positive proportional  vector  solution of \eqref{1.1}, see Remark~\ref{re3.4} later.
\end{remark}

The proof for  Theorems  \ref{Th1} and \ref{Th1.4} can be divided into  two parts.
In the first part, we establish the existence of a positive proportional vector solution with the form $(k_1w,\tau k_1w)$ to \eqref{1.1}. In this case,  we see
\begin{equation}\label{KC}
(\mu_1+\beta\tau^p)k_1^{2p-2}=1=(\mu_2\tau^{2p-2}+\beta\tau^{p-2})k_1^{2p-2}
\end{equation}
 and $\tau$ satisfies
\begin{equation}
\mu_1+\beta\tau^p-\mu_2\tau^{2p-2}-\beta\tau^{p-2}=0.
\end{equation}
So we   need to
investigate the solvability  of the following equation
\begin{equation}\label{1.10}
g(\tau):=\mu_1+\beta\tau^p-\mu_2\tau^{2p-2}-\beta\tau^{p-2}=0,\quad \tau \in D.
\end{equation}

In the second part, we  prove that any positive  proportional vector solution, got in the first part, is
non-degenerate.   Our method is inspired by  \cite{bdw} and  \cite{pw} where a special case $s=1,N=3$ and $p=2$ was studied. We will convert the study on the non-degeneracy of the solutions for system \eqref{1.1} into that of a single equation by using linearization and spectral analysis. As Peng and Wang \cite{pw} did, here
we have to prove that $\tilde{f}(\tilde{\beta}) \neq \lambda_k~(\forall k\in N^+)$, where $\tilde{f}(\tilde{\beta})$ will be defined in
\eqref{3.2} and \eqref{3.3}, and $\lambda_k~(k\in N^+)$ are the eigenvalues of  the weighted
eigenvalue problem  $(-\Delta)^s u +u= \lambda w^{2p-2}u$.  However, compared to \cite{pw}, we will encounter  more difficulties. On one hand, since $p$ is more general in system \eqref{1.1}, we can not write out the explicit expression on $\tau_0, k_1$, which  makes  our discussion    more complicated. On the other hand, in \cite{pw}, Peng and Wang obtained  the non-degeneracy by proving that the corresponding $\tilde{f}(\tilde{\beta})$ is
monotone about $\tilde{\beta}$.  But in our case, we only get the same result for the case $\tilde{\beta}<0$. But, for the case $\tilde{\beta}>0$,  we have to discuss  $p$ in three cases:
$1<p<2, p=2, 2<p<2^*_s/2$. In each case,  we should carry out  some tedious and preliminary analysis  to get the non-degeneracy result in some ranges of  $\mu_1,\mu_2, p, \tilde{\beta}$.

To verify Theorem \ref{th1.6},  we first prove that \eqref{1.1} has  a least energy solution of the form $(kw, k\tau w)$. We observe  that  $(kw, k\tau w)$ is a least energy solution of \eqref{1.1} if and only if $S_{\mu_1,\mu_2}$ can be obtained by $(kw, k\tau w)$, which help us to reduce the  problem to considering a minimization problem $\min_{\tau\ge 0}f(\tau)$. Then, we prove that any positive  least energy solution to \eqref{1.1} must be proportional and the minimum point $\tau_{min}$ of $f(\tau)$ must be unique.

This paper is organized as follows. In Section 2, we will prove Theorems \ref{Th1} and \ref{Th1.4}. The  proof for Theorem \ref{th1.6}  will be provided  in Section \ref{s3}.
The analysis on  minimum or maximum points of $f(\tau)$ will  be given in the Appendix.

\section{ Existence and non-degeneracy of proportional solutions}

In this section, we  prove  Theorems \ref{Th1} and \ref{Th1.4}. To this end,   we first consider the following system about $\tau$:
\begin{equation}\label{2.1}
\left\{%
\begin{array}{ll}
p\tau^{2p-2}-2\tilde{\beta}\tau^p=\mu(2-p),\vspace{2mm}\\
\tilde{g}(\tau):=\mu+\tilde{\beta}\tau^p-\tilde{\beta}\tau^{p-2}-\tau^{2p-2}=0.
\end{array}%
\right.
\end{equation}

\begin{lemma}\label{lm2.1}
Assume that $\mu>1, \tilde{\beta}>0.$ Then system \eqref{2.1} admits no solutions in $\tilde{D}$, provided    one of the following conditions holds:\\
$(B_1)~ p>2, ~0<\tilde{\beta}\leq(p-1)\mu;$\\
$(B_2)~p>2,~\mu \geq \frac{1}{2}(\frac{p}{p-1})^{p-1},~\tilde{\beta }\geq (p-1)\mu;$\\
$(B_3)~p>2,~1<\mu<\frac{1}{2}(\frac{p}{p-1})^{p-1}, (p-1)\mu\leq \tilde{\beta}\leq\tilde{\beta}_0,
\hbox{or}~\max\Big\{\tilde{\beta}_1,(p-1)\mu\Big\}\leq \tilde{\beta};$\\
$(B_4)~1<p<2, ~\tilde{\beta}\geq (p-1)\mu^\frac{p-2}{2(p-1)};$\\
$(B_5)~1<p<2, ~1<\mu<\frac{p}{2-p},~0<\tilde{\beta}\leq \min\Big\{\frac{p-\mu(2-p)}{2}, ~2(p-1)(2-p)^\frac{2-p}{2(p-1)}(\frac{\mu}{p})^\frac{p}{2(p-1)}\Big\};~$\\
$(B_6)~1<p<2,\max\Big\{\frac{p-\mu(2-p)}{2},~0\Big\}<\tilde{\beta}<(p-1)\mu^\frac{p-2}{2(p-1)}$;\\
$(B_7)~p=2, ~\tilde{\beta} >0,$\\
where $\tilde{D}$ is defined in \eqref{1.7'}, and  $0<\tilde{\beta}_0<\tilde{\beta}_1$ are the roots of $\frac{2(p-1)}{p}\mu-H_1(\tilde{\beta})=0$.
\end{lemma}

\begin{proof} Set
$$F(\tau):=p\tau^{2p-2}-2\tilde{\beta}\tau^p,$$
 and
  $$G(\tau):=\tilde{g}(\tau)+\frac{1}{p}(F(\tau)-\mu(2-p)).$$
Then
$$G(\tau)=\frac{2(p-1)}{p}\mu+\frac{\tilde{\beta}(p-2)}{p}\tau^p-\tilde{\beta}\tau^{p-2}.
$$
 We divide the proof into three cases: $2<p, 1<p<2$ and $p=2.$

Case I: \,\,$\tilde{\beta}>0, p>2$.

 In this case,   $G(\tau)$ gets its
minimum at $ \tau^G_{min}=1$ and $G(1)=\frac{2}{p}[(p-1)\mu-\tilde{\beta}]$.

  If $0<\tilde{\beta}<(p-1)\mu$,  then $G(\tau^G_{min})>0.$ So  \eqref{2.1} has no solution in $\tilde{D}$.

   If $\tilde{\beta}=(p-1)\mu$, then $G(1)=0~\hbox{but}~\tilde{g}(1)=\mu-1>0$. So \eqref{2.1} has no solution in $\tilde{D}$.

 At last,  if $\tilde{\beta}>(p-1)\mu$, then  $F(1)<\mu(2-p)$ since $\mu>1$ and $2<p$.
So $F(\tau)=\mu(2-p)$ has two solutions $0<\tau_1<\tau_2$. It follows from $F(\tau_i)=\mu(2-p)$ that $p\tau_i^{2p-2}<2\tilde{\beta}\tau_i^p$ and $\mu(p-2)<2\tilde{\beta}\tau_i^p$, which implies
$$
\Big[\frac{\mu(p-2)}{2\tilde{\beta}}\Big]^\frac{1}{p}<\tau_1<\tau_2
<\Big[\frac{2\tilde{\beta}}{p}\Big]^\frac{1}{p-2}.
$$
  But from
$\tilde{\beta}>(p-1)\mu$, we know  $\tau^F_{min}=(\frac{\tilde{\beta}}{p-1})^\frac{1}{p-2}>1$. Therefore, by the graph of $F(\tau)$, we see
\begin{equation}\label{2.51}
\Big[\frac{\mu(p-2)}{2\tilde{\beta}}\Big]^\frac{1}{p}<\tau_1<1,\,\,
1<\Big(\frac{\tilde{\beta}}{p-1}\Big)^\frac{1}{p-2}<\tau_2
<\Big[\frac{2\tilde{\beta}}{p}\Big]^\frac{1}{p-2}.
\end{equation}
Since we only consider the solutions in $(1,+\infty)$, we can complete our proof if we can prove $G(\tau_2)\neq0$.
  $G(\tau)$  increases  strictly in $[1, +\infty],$  hence
  \begin{eqnarray*}
  G(\tau_2) &>&G\Big((\frac{\tilde{\beta}}{p-1})^\frac{1}{p-2}\Big)=\frac{2(p-1)}{p}\mu-\Big[\frac{1}{p-1}\tilde{\beta}^2-\frac{p-2}{p}(p-1)
^{-\frac{p}{p-2}}\tilde{\beta}^\frac{2(p-1)}{p-2}\Big]\\
&=&\frac{2(p-1)}{p}\mu-H_1(\tilde{\beta}).
 \end{eqnarray*}
 Using the definition of $H_1(\tilde{\beta})$, we conclude that $H_1(\tilde{\beta})$ attains its maximum at $ \tilde{\beta}_{max}=p^\frac{p-2}{2}(p-1)^\frac{4-p}{2}$ and $~H_1(\tilde{\beta}_{max})=(\frac{p}{p-1})^{p-2}.$ Now we have two subcases: If $\mu \geq \frac{1}{2}(\frac{p}{p-1})^{p-1},$ then $\frac{2(p-1)}{p}\mu\geq H_1(\tilde{\beta}_{max}).$ That is,
$G(\tau_2) >G((\frac{\tilde{\beta}}{p-1})^\frac{1}{p-2})\geq 0,$ which implies that $\tilde{g}(\tau_2) \neq 0$ and \eqref{2.1} has no solution in $\tilde{D}$; If $1<\mu<\frac{1}{2}(\frac{p}{p-1})^{p-1},$ then  $\frac{2(p-1)}{p}\mu< H_1(\tilde{\beta}_{max}).$
Hence  there exist $\tilde{\beta}_0, \tilde{\beta}_1$ such that  $\frac{2(p-1)}{p}\mu\geq H_1(\tilde{\beta})$ when $0<\tilde{\beta} \leq \tilde{\beta}_0$ or $\tilde{\beta} \geq \tilde{\beta}_1$.  So
$G(\tau_2) >G((\frac{\tilde{\beta}}{p-1})^\frac{1}{p-2})\geq 0,$
which also implies that $\tilde{g}(\tau_2) \neq 0$.
As a result, if $(p-1)\mu\leq \tilde{\beta}\leq \tilde{\beta}_0~\hbox{or}~\max\{(p-1)\mu, \tilde{\beta}_1\}\leq\tilde{\beta}$, then \eqref{2.1} has no solution in $\tilde{D}$.\\

Case II:\,\,$ \tilde{\beta} >0, 1<p<2$

In this case, $F(\tau^F_{max})=(2-p)(\frac{p-1}{\tilde{\beta}})^\frac{2(p-1)}{2-p}$.

If $(p-1)\mu^\frac{p-2}{2(p-1)} <\tilde{\beta}
 $, then $F(\tau)\leq F(\tau^F_{max})<\mu(2-p)$. So  \eqref{2.1} has no solution in $\tilde{D}$.

 If $\tilde{\beta}=(p-1)\mu^\frac{p-2}{2(p-1)}$, then $F(\tau)=\mu(2-p)$ has unique solution $ \tau_0=(\frac{p-1}{\tilde{\beta}})^\frac{1}{2-p}.$  Direct calculation gives
 \begin{eqnarray*}
 G(\tau_0)&=&\frac{2(p-1)}{p}\mu-\frac{2-p}{p}(p-1)-(p-1)\mu^\frac{p-2}{p-1}\\
 &>&\frac{2(p-1)}{p}\mu-\frac{2-p}{p}(p-1)-(p-1)
 =\frac{2(p-1)}{p}(\mu-1)>0.
 \end{eqnarray*}
 So \eqref{2.1} has no solution in $\tilde{D}$.

At last,  if $0< \tilde{\beta}<(p-1)\mu^\frac{p-2}{2(p-1)}$, then the equation  $F(\tau)=\mu(2-p)$ has two solutions $\tau_1, \tau_2.$ This case is more complicated. From the definition of $F(\tau)$, we know that $F(\tau)$  increases strictly in $[0, (\frac{p-1}{\tilde{\beta}})^\frac{1}{2-p}]$ and  decreases strictly in $[(\frac{p-1}{\tilde{\beta}})^\frac{1}{2-p}, +\infty]$. Similarly, we deduce
 that if $1<\mu<\frac{p}{2-p}~\hbox{and}~0<\tilde{\beta}\leq\frac{p-\mu(2-p)}{2}$, then $F(1)\geq \mu(2-p).$ Proceeding as we prove  \eqref{2.51}, we find
$$
\Big[\frac{\mu(2-p)}{p}\Big]^\frac{1}{2(p-1)}<\tau_1\leq 1,\,\,  1<\Big(\frac{p-1}{\tilde{\beta}}\Big)^\frac{1}{2-p}<\tau_2<\Big[\frac{p}{2\tilde{\beta}}\Big]^\frac{1}{2-p}.
$$

Similar  to case $p>2$, we only need to prove $G(\tau_1)\neq0$. Direct computation yields that   $G(\tau)$  increases strictly  in $[0, 1]$ and decreases strictly  in $[1, +\infty]$. If $\tilde{\beta}\leq 2(p-1)(2-p)^\frac{2-p}{2(p-1)}(\frac{\mu}{p})^\frac{p}{2(p-1)},$ then $ G\Big(\big[\frac{\mu(2-p)}{p}\big]^\frac{1}{2(p-1)}\Big)\geq 0$. Since $G(1) \geq \mu-1>0$, we conclude  that $G(\tau_1)>0$; But if $\max\Big\{\frac{p-\mu(2-p)}{2},~0\Big\}<\tilde{\beta}<(p-1)\mu^\frac{p-2}{2(p-1)}$, then
 $\max\{(\frac{\mu(2-p)}{p})^\frac{1}{2(p-1)}, 1 \}<\tau_1<(\frac{p-1}{\tilde{\beta}})^\frac{1}{2-p}<\tau_2<(\frac{p}{2\tilde{\beta}})^\frac{1}{2-p},$
 which implies $p\tau^{2p-2}-2\tilde{\beta}\tau^p=\mu(2-p)$ has no solution in $(0,1)$, and
 hence \eqref{2.1} has no solution in $\tilde{D}$.\\

Case  III:\,\, $p=2$

 \eqref{2.1} can be written as
$$\left\{%
\begin{array}{ll}
\tau^2-\tilde{\beta}\tau^2=0,\\
\mu+\tilde{\beta}\tau^2-\tilde{\beta}-\tau^2=0.
\end{array}%
\right.$$
It is easy to see that the above system has no solutions in $\tilde{D}$.
\end{proof}

In the following, we will consider the case  $\tilde{\beta}<0$.

Define
$$
l(\tilde{\beta})=\frac{\mu(2p-1)+\tilde{\beta}(p-1)\tau^p-\tilde{\beta} p \tau^{p-2}}{\mu+\tilde{\beta} \tau^p},
$$
where $\tau\in \tilde{D}~\hbox{and}~\tilde{\beta}$ satisfy
$$
\mu+\tilde{\beta}\tau^p-\tau^{2p-2}-\tilde{\beta}\tau^{p-2}=0,\quad \mu+\tilde{\beta}\tau^p>0,\quad \tau^{2p-2}+\tilde{\beta}\tau^{p-2}>0.
$$

 \begin{lemma}\label{npro3}
Assume that $0<s<1, 1<p<\frac{2^*_s}{2}, \mu>1$. Then $l(\tilde{\beta})$  decreases strictly in $(-\sqrt{\mu},0).$
 \end{lemma}
\begin{proof} Differentiating with respect to $ \tilde{\beta}$ on both sides of the equation
 $\mu+\tilde{\beta}\tau^p-\tau^{2p-2}-\tilde{\beta}\tau^{p-2}=0$, we get
 $$
 \tau^\prime(\tilde{\beta})=\frac{\tau(\tau^2-1)}{2(p-1)\tau^p-\tilde{\beta} p\tau^2+\tilde{\beta}(p-2)}.
 $$
 So
  \begin{eqnarray*}
 l^\prime(\tilde{\beta})&=&\ds\frac{\big[(p-1)\tau^p-p\tau^{p-2}+\tilde{\beta} p\tau^{p-3}((p-1)\tau^2-(p-2))\tau^\prime(\tilde{\beta})\big]\big(\mu+\tilde{\beta} \tau^p\big)}{(\mu+\tilde{\beta} \tau^p)^2}\\
 &&-\ds\frac{\tau^p\big[\mu(2p-1)+\tilde{\beta}(p-1)\tau^p-\tilde{\beta} p \tau^{p-2}\big]}{(\mu+\tilde{\beta} \tau^p)^2}\\
 &&+\ds\frac{\tilde{\beta} p\tau^{p-1}\big[\mu(2p-1)+\tilde{\beta}(p-1)\tau^p-\tilde{\beta} p \tau^{p-2}\big]\tau^\prime(\tilde{\beta})}{(\mu+\tilde{\beta} \tau^p)^2}\\
 &=&\ds\frac{1}{(\mu+\tilde{\beta} \tau^p)^2}\Big[-\mu p\tau^{p-2}(\tau^2+1)+\tilde{\beta} p\tau^{p-2}(\tau^2-1)\frac{2\tilde{\beta}\tau^p-\mu p\tau^2+\mu(2-p)}{2(p-1)\tau^p-\tilde{\beta} p\tau^2+\tilde{\beta}(p-2)}\Big]\\
 &=&\ds\frac{\tau^{p-2}}{(\mu+\tilde{\beta} \tau^p)^2}\frac{\tilde{\beta} p(\tau^2-1)\big(2\tilde{\beta}\tau^p\big)}{2(p-1)\tau^p-\tilde{\beta} p\tau^2+\tilde{\beta}(p-2)}\\
 &&-\ds\frac{\big[\mu p\tau^2+\mu(2-p)\big]-\mu p(\tau^2+1)\big[2(p-1)\tau^p-\tilde{\beta} p\tau^2+\tilde{\beta}(p-2)\big]}{2(p-1)\tau^p-\tilde{\beta} p\tau^2+\tilde{\beta}(p-2)}\\
 &=&\ds\frac{2p\tau^p}{(\mu+\tilde{\beta} \tau^p)^2}\frac{\big[\tilde{\beta}^2-\mu(p-1)\big]\tau^p-\big[\tilde{\beta}^2+\mu(p-1)\big]\tau^{p-2}+2\mu\tilde{\beta}}{2(p-1)\tau^p-\tilde{\beta} p\tau^2+\tilde{\beta}(p-2)}.
\end{eqnarray*}
 Since $\mu+\tilde{\beta}\tau^p-\tau^{2p-2}-\tilde{\beta}\tau^{p-2}=0, \tilde{\beta}<0$ and $0<\tau\in \tilde{D}$, we see $\tau>1,~\hbox{and}~ \mu-\tau^{2p-2}>0$, which combined  with  $\mu+\tilde{\beta}\tau^p>0,\,\tau^{2p-2}+\tilde{\beta}\tau^{p-2}>0$ yields
  \begin{equation}\label{2.3}
 \min\Big\{(\frac{\mu}{|\tilde{\beta}|})^\frac{1}{p},\mu^\frac{1}{2(p-1)}\Big\}>\tau >\max\Big\{1,|\tilde{\beta}|^\frac{1}{p}\Big\}.
  \end{equation}

 Set $M(\tau)=2(p-1)\tau^p-\tilde{\beta} p\tau^2+\tilde{\beta}(p-2).$  Then $M(\tau)$ increases in $[1,+\infty)$  and   $M(\tau)>M(1)=2(p-1)-2\tilde{\beta}>0$. To get some $\tau$ satisfying \eqref{2.3}, we know that   $\tilde{\beta}$ should satisfy $ \tilde{\beta}^2<\mu.$  Set
 $$
 T(\tilde{\beta})=
  \big[\tilde{\beta}^2-\mu(p-1)\big]\tau^p-\big[\tilde{\beta}^2+\mu(p-1)\big]\tau^{p-2}+2\mu\tilde{\beta}.
  $$

  If $p\geq 2$, we see from  $ \tilde{\beta}^2<\mu$ that $ \tilde{\beta}^2\leq\mu(p-1).$  Using the facts that $\tau>0, \mu>1 $ and $\tilde{\beta}<0$, we conclude   $T(\tilde{\beta})<0$.  So $l(\tilde{\beta})$   decreases strictly in $(-\sqrt{\mu},0).$

  If $1<p<2$ and $\tilde{\beta}^2\leq\mu(p-1)$, similar to the  case $p\geq 2$, we obtain  that  $l(\tilde{\beta})$  decreases strictly in $(-\sqrt{\mu(p-1)},0)$.

    At last, if $1<p<2$ and $\mu>\tilde{\beta}^2>\mu(p-1)$, it follows from  $\tau>1$ that $\tau^\prime(\tilde{\beta})>0$, which combined  with $\mu+\tilde{\beta}\tau^p-\tau^{2p-2}-\tilde{\beta}\tau^{p-2}=0$ yields that
    $$
    T^\prime(\tilde{\beta})=2\tau^{2p-2}+ \Big\{\big[\tilde{\beta}^2-\mu(p-1)\big]p\tau^{p-1}+\big[\tilde{\beta}^2
  +\mu(p-1)\big](2-p)\tau^{p-3}\Big\}\tau^\prime(\tilde{\beta})>0.
  $$
   Thus $T(\tilde{\beta})$ increases in $[-\sqrt{\mu}, -\sqrt{\mu(p-1)}]$. Direct computation verifies   $T(-\sqrt{\mu(p-1)})<0$.
 So $T(\tilde{\beta})<0$ in $[-\sqrt{\mu}, -\sqrt{\mu(p-1)}]$. As a consequence,  $l(\tilde{\beta})$  decreases strictly in $[-\sqrt{\mu}, -\sqrt{\mu(p-1)}]$.

  \end{proof}

\begin{lemma}\label{lm2.3} Suppose that $p\in(1,+\infty)\setminus\{2\}, \beta \in (-\sqrt{\mu_1\mu_2},0)\cup(0, +\infty)~~ \hbox{or}~~
p=2, \beta \in (-\sqrt{\mu_1\mu_2},0)\cup(0, \mu_2)\cup(\mu_1, +\infty)$. Then there exists    $\tau_0\in D$ such that $g(\tau_0)=0, \mu_1+\beta\tau_0^p>0$  and $\mu_2\tau_0^{2p-2}+\beta\tau_0^{p-2}>0$, where $g(\tau)$ is defined in \eqref{1.10}.
\end{lemma}
\begin{proof}  The proof can be divided  into two cases: $\beta>0, \beta<0$.

Case I: $\beta>0$.

In this case, we consider the following three subcases:

1)\,\, $1<p<2$.

We see $2(p-1)>0, p-2<0$ and for any fixed $\beta>0$,  $\lim\limits_{\tau \to 0^+} g(\tau)=-\infty$ and $g(1)=\mu_1-\mu_2>0$. So there exists  $\tau_0(\beta)\in D$
such that $g(\tau_0)=0, \mu_1+\beta\tau_0^p>0 ~\hbox{and}~ \mu_2\tau_0^{2p-2}+\beta\tau_0^{p-2}>0$.

2)\,\,  $p>2$.

We have $2p-2>p> p-2>0$ and for any fixed $\beta>0$,  $g(1)>0,~~ \hbox{and}~~
\lim\limits_{\tau \to +\infty}g(\tau)=-\infty.$  Hence there exists $\tau_0(\beta)\in D$
such that $g(\tau_0)=0, \mu_1+\beta\tau_0^p>0 ~\hbox{and}~ \mu_2\tau_0^{2p-2}+\beta\tau_0^{p-2}>0$.

3)\,\, $p=2$.

  We find $g(\tau)=\mu_1-\beta+(\beta-\mu_2)\tau^2$. If $\beta>\mu_1,~~ \hbox{or}~~ 0<\beta<\mu_2$, then  there exists  $\tau_0(\beta)\in D$ such that $g(\tau_0)=0, \mu_1+\beta\tau_0^p>0 ~\hbox{and}~ \mu_2\tau_0^{2p-2}+\beta\tau_0^{p-2}>0$.

Case II:  $\beta <0$.

Firstly,  to guarantee  that $\mu_1+\beta\tau_0^p>0 ~\hbox{and}~ \mu_2\tau_0^{2p-2}+\beta\tau_0^{p-2}>0$, we need  $(\frac{|\beta|}{\mu_2})^\frac{1}{p}<\tau_0<(\frac{\mu_1}{|\beta|})^\frac{1}{p},$ which implies
$\beta \in (-\sqrt{\mu_1\mu_2}, 0).$ Secondly, since  $\beta \in (-\sqrt{\mu_1\mu_2}, 0)$, we see $g((\frac{|\beta|}{\mu_2})^\frac{1}{p})=\frac{\mu_1\mu_2-\beta^2}{\mu_2}>0$
 and $g((\frac{\mu_1}{|\beta|})^\frac{1}{p})=-\frac{1}{|\beta|}(\frac{\mu_1}{|\beta|})^\frac{p-2}{p}(\mu_1\mu_2-\beta^2)<0$. Thus there exists    $\tau_0(\beta)\in D$ such that $g(\tau_0)=0$.

 Therefore, when $\beta \in (-\sqrt{\mu_1\mu_2}, 0)$, there exists  $\tau_0(\beta)\in D$
such that $g(\tau_0)=0, \mu_1+\beta\tau_0^p>0$ and $\mu_2\tau_0^{2p-2}+\beta\tau_0^{p-2}>0$.
\end{proof}

\begin{proof} [\bf Proof of  Theorem \ref{Th1} and  Theorem \ref{Th1.4}:] We first prove the existence of positive solutions with the form $(kw, \tau kw)$  to \eqref{1.1}.  Put $(U,V):=(kw, \tau kw)$ into  system \eqref{1.1}, we have
$$
\left\{%
\begin{array}{ll}
(-\Delta)^s w +w=(\mu_1+\beta\tau^p)|k|^{2p-2}w^{2p-1},~~x\in \R^N,\vspace{2mm}\\
(-\Delta)^s w +w=(\mu_2\tau^{2p-2}+\beta\tau^{p-2})|k|^{2p-2}w^{2p-1},~~x\in \R^N.\\
\end{array}%
\right.
$$
According to   Lemma \ref{lm2.3},  we can find  $\tau_0\in D$  such that
$$\mu_1+\beta\tau_0^p=\mu_2\tau_0^{2p-2}+\beta\tau_0^{p-2}>0.
$$
Thus take $k_1>0$ satisfying $k_1^{2p-2}=(\mu_1+\beta\tau_0^p)^{-1}$, then
$$
(\mu_1+\beta\tau_0^p)k_1^{2p-2}=1, (\mu_2\tau_0^{2p-2}+\beta\tau_0^{p-2})k_1^{2p-2}=1,
$$
 which implies that  $(k_1w, \tau_0k_1w)$ is a radial positive  solution of \eqref{1.1}.

Next, we prove that any positive proportional vector solution got above is  non-degenerate.

Let  $(U,V):=(k_1w, \tau_0k_1w)\in H^s(\R^N)\times H^s(\R^N)$ is a positive  solution  of  the  system \eqref{1.1}. In system \eqref{1.1}, making  a change $(u,v) \to (\mu_2^{-\frac{1}{2p-2}}u,
 \mu_2^{-\frac{1}{2p-2}}v)$, we see
\begin{equation}\label{3.1}
\left\{%
\begin{array}{ll}
(-\Delta)^s u +u= \mu|u|^{2p-2}u+\tilde{\beta} |v|^p|u|^{p-2}u,~~x\in \R^N,\vspace{2mm}\\
(-\Delta)^s v +v= |v|^{2p-2}v+\tilde{\beta} |u|^p|v|^{p-2}v,~~x\in \R^N,
\end{array}
\right.
\end{equation}
where $\mu=\frac{\mu_1}{\mu_2}, \tilde{\beta}=\frac{\beta}{\mu_2}$.
So to complete the proof of  Theorems \ref{Th1} and \ref{Th1.4}, it suffices to prove the same conclusion for system \eqref{3.1}.

 Consider the weighted eigenvalue problem in $ \lambda: (-\Delta)^s u +u= \lambda w^{2p-2}u$. It follows from \cite{fls} that this equation  has a sequence of eigenvalues
$1 = \lambda_1 < \lambda_2 = \lambda_3 = \cdots=\lambda_{N+1}=2p-1< \lambda_{N+2}\leq\cdots$ with associated eigenfunctions $\Phi_k$ satisfying
$\int_{\R^N}w^{2p-2}\Phi_k\Phi_m=0$
for $k\neq m$. For $\Phi_k$ with $k = 2, 3,\cdots,N+1$, we may take them as $\frac{\partial w}{\partial x_1},
\frac{\partial w}{\partial x_2},\cdots,\frac{\partial w}{\partial x_N}.$
 The  linearization  of  \eqref{3.1} at $(\mu_2^{-\frac{1}{2p-2}}U, \mu_2^{-\frac{1}{2p-2}}V)$ is
$$
\left\{%
\begin{array}{ll}
(-\Delta)^s \varphi +\varphi=w^{2p-2}(a\varphi+b\psi),~~x\in \R^N,\vspace{2mm}\\
(-\Delta)^s \psi +\psi=w^{2p-2}(c\psi+b\varphi),~~x\in \R^N,
\end{array}%
\right.
$$
where
\begin{eqnarray*}
a(\tau_0)&=&[\mu(2p-1)+\tilde{\beta}(p-1)\tau_0^p]k_1^{2p-2},\\
b(\tau_0)&=&\tilde{\beta} p \tau_0^{p-1}k_1^{2p-2},\\
c(\tau_0)&=&[(2p-1)\tau_0^{2p-2}+\tilde{\beta}(p-1)\tau_0^{p-2}]k_1^{2p-2}.
\end{eqnarray*}
Let $\gamma_{\pm}=\frac{(a-c)}{2b}\pm\frac{\sqrt{(c-a)^2+4b^2}}{2b} $ be the solutions of equation
$$
c\gamma-b=\gamma(a-b\gamma).
$$
Now, we complete the proof of Theorem \ref{Th1.4}.

If $\tilde{\beta}<0$, by direct computation, we obtain  $a-b\gamma_{-}=2p-1.$
So
$$
(-\Delta)^s(\varphi-\gamma_{-}\psi)+(\varphi-\gamma_{-}\psi)=(2p-1)w^{2p-2}(\varphi-\gamma_{-}\psi),
$$
and
$$
\varphi-\gamma_{-}\psi=\sum\limits_{j=2}^{N+1}\alpha_j\Phi_j.
$$
Thus,
$$
(-\Delta)^s\psi+\psi =(b\gamma_{-}+c)\psi
w^{2p-2}+\sum\limits_{j=2}^{N+1}b\alpha_j\Phi_jw^{2p-2}.
$$

Set $\psi=\sum\limits_{j=1}^\infty \Gamma_j\Phi_j$
and
\begin{equation}\label{3.2}
\tilde{f}(\tilde{\beta})=b\gamma_{-}+c.
 \end{equation}
 By direct computation, we have  $\tilde{f}(\tilde{\beta})=\frac{\mu(p-1)-\tilde{\beta}\tau_0^p+p\tau_0^{2p-2}}{\mu+\tilde{\beta}\tau_0^p}.$  Since $\mu+\tilde{\beta}\tau_0^p-\tau_0^{2p-2}-\tilde{\beta}\tau_0^{p-2}=0$,
we obtain that $\tilde{f}(\tilde{\beta})=l(\tilde{\beta})
=\frac{\mu(2p-1)+\tilde{\beta}(p-1)\tau_0^p-\tilde{\beta} p \tau_0^{p-2}}{\mu+\tilde{\beta} \tau_0^p}$.
From Lemma \ref{npro3},  there exists a decreasing sequence $\{\tilde{\beta_k}\}$ such that $\tilde{f}(\tilde{\beta_k})=\lambda_k$ and
$\tilde{f}(\tilde{\beta})\neq \lambda_k$ for any $\tilde{\beta}\in (-\sqrt{\mu},0)\setminus\{\tilde{\beta_k}\}$ and  $k=1,2,\cdots$. Using
orthogonality, we see  that $\Gamma_j=0$ for $j\neq 2, 3,\cdots,N+1$ and $\Gamma_j=\frac{b\alpha_j}{2p-1-\tilde{f}(\tilde{\beta})}
$ for $j= 2, 3,\cdots,N+1$.
Thus, the kernel at $(U, V)$ is given by span$\{((\gamma_-\frac{b}{2p-1-\tilde{f}(\tilde{\beta})}+1)\Phi_k,\frac{b}{2p-1-\tilde{f}(\tilde{\beta})}\Phi_k)|
k=2,3, \cdots,N+1\}$, a $N$-dimensional space. Taking $\theta(\tilde{\beta})=\gamma_-+\frac{2p-1-\tilde{f}(\tilde{\beta})}{b}$, we can check $\theta(\tilde{\beta})\neq 0$. Therefore, there exists a decreasing sequence $\{\beta_k\}$ such that  when $\beta \in (-\sqrt{\mu_1\mu_2},0)\setminus\{\beta_k\}$, the conclusion of Theorem \ref{Th1.4} is true.

At last, we  finish proving   Theorem \ref{Th1}.  If $\tilde{\beta}>0$, we can check  $a-b\gamma_{+}=2p-1.$  So
$$
  (-\Delta)^s(\varphi-\gamma_{+}\psi)+(\varphi-\gamma_{+}\psi)=(2p-1)w^{2p-2}(\varphi-\gamma_{+}\psi),
  $$
and
$$
\varphi-\gamma_{+}\psi=\sum\limits_{j=2}^{N+1}\alpha_j\Phi_j.
$$
Thus,
$$
(-\Delta)^s\psi+\psi =(b\gamma_{+}+c)\psi
w^{2p-2}+\sum\limits_{j=2}^{N+1}b\alpha_j\Phi_jw^{2p-2}.
$$

Set $\psi=\sum\limits_{j=1}^\infty \Gamma_j\Phi_j$
and
\begin{equation}\label{3.3}
\tilde{f}(\tilde{\beta})=b\gamma_{+}+c.
\end{equation}
 By direct computation, we have
 $\tilde{f}(\tilde{\beta})=\frac{\mu(2p-1)+\tilde{\beta}(p-1)\tau_0^p-\tilde{\beta} p \tau_0^{p-2}}{\mu+\tilde{\beta} \tau_0^p}$.

Claim I:~~ $\tilde{f}(\tilde{\beta})\neq 1$  if any one of $(A_i)$ ($i=1,2,\cdots,7$) holds.

We assume that $\tilde{f}(\tilde{\beta})=1$. Then there exist $\tilde{\beta}$ and $\tau_0\in \tilde{D}$ such that
 \begin{equation}
 \left\{%
\begin{array}{ll}
\mu(2p-1)+\tilde{\beta}(p-1)\tau_0^p-\tilde{\beta} p \tau_0^{p-2}=\mu+\tilde{\beta} \tau_0^p,\vspace{2mm}\\
\ds \mu+\tilde{\beta}\tau_0^p-\tilde{\beta}\tau_0^{p-2}-\tau_0^{2p-2}=0.\\
\end{array}%
\right.
\end{equation}
That is,
 \begin{equation}\left\{%
\begin{array}{ll}
p\tau_0^{2p-2}-2\tilde{\beta}\tau_0^p=\mu(2-p),\vspace{2mm}\\
\mu+\tilde{\beta}\tau_0^p-\tilde{\beta}\tau_0^{p-2}-\tau_0^{2p-2}=0,\vspace{2mm}\\
\tau_0\in \tilde{D},
\end{array}%
\right.\end{equation}
which contradicts to Lemma \ref{lm2.1}, since the conditions given in Lemma \ref{lm2.1} correspond to those given in \eqref{1.8} respectively. So the claim I is true.

Claim II: ~~$\tilde{f}(\tilde{\beta})<2p-1$.

Assume  that $\ds \tilde{f}(\tilde{\beta}) \geq 2p-1.$ Then   $0>-\tilde{\beta} p \tau_0^{p-2}\geq p\tilde{\beta}\tau_0^p>0$, which is impossible, since $p, \tau_0, \tilde{\beta}>0$. So the claim II is true.

 Claim I and Claim II imply  that $\tilde{f}(\tilde{\beta})\neq \lambda_k$ for any $k=1,2,\cdots$.  Proceeding as we prove Theorem \ref{Th1.4}, we can complete the proof of Theorem \ref{Th1}.
\end{proof}

\section{the least energy solutions}\label{s3}

  \begin{lemma}\label{lm3.1}
  Assume that $\beta>0$ and $1<p<\frac{2^*_s}{2}, \mu_1,\mu_2>0$. Then

 $(1)~~ S_{\mu_1,\mu_2}=f(\tau_{min})S$

 $(2)~~ S_{\mu_1,\mu_2}$ can be obtained by $(w_{x_0}, \tau_{min}w_{x_0})$ for all $x_0 \in \R^N$, where  $\tau_{min}\ge 0$ is a minimum point of $f(\tau)$ in $[0, +\infty)$ satisfying
 $$
 \tau(\mu_1+\beta\tau^p-\mu_2\tau^{2p-2}-\beta\tau^{p-2})=0,
 $$
 $S, S_{\mu_1,\mu_2}, f(\tau) $ are defined respectively in \eqref{S},\eqref{Smu1mu2} and \eqref{ftau}.

  \end{lemma}

\begin{proof}
Since $w$ is the ground state of \eqref{1.3}, we see
\begin{equation}\label{2.4}
S_{\mu_1,\mu_2} \leq \frac{\ds\int_{\R^N}(|\xi|^{2s}+1)(1+\tau_{min}^2)|\hat{w}|^2}
{[(\mu_1+2\beta\tau_{min}^p+\mu_2\tau_{min}^{2p})\ds\int_{\R^N}|w|^{2p}]^{^\frac{1}{p}}}
=f(\tau_{min})S.
\end{equation}

Let $(u_n, v_n)$ be a minimizing sequence for $S_{\mu_1,\mu_2}$, and   $\tau_n$ be a positive constant such that
\begin{equation}\label{2.5}\tau_n^{2p}\ds\int_{\R^N}|u_n|^{2p}=\ds\int_{\R^N}|v_n|^{2p}.
\end{equation}

Set $z_n:=\frac{1}{\tau_n}v_n$. By Young's inequality, we have
\begin{equation}\label{2.6}
\ds\int_{\R^N}|u_n|^p|z_n|^p \leq \frac{1}{2}\ds\int_{\R^N}|u_n|^{2p}+\frac{1}{2}\ds\int_{\R^N}|z_n|^{2p}=\ds\int_{\R^N}|u_n|^{2p}.
\end{equation}
Therefore,
\begin{equation}\label{2.7}\begin{array}{ll}
S_{\mu_1,\mu_2} +o_n(1)&=\frac{\ds\int_{\R^N}(|\xi|^{2s}+1)(|\hat{u_n}|^2+|\hat{v_n}|^2)}
{\Big[\ds\int_{\R^N}(\mu_1|u_n|^{2p}+2\beta|u_n|^p|v_n|^p+\mu_2|v_n|^{2p})\Big]^{^\frac{1}{p}}}\\
&=\frac{\ds\int_{\R^N}(|\xi|^{2s}+1)(|\hat{u_n}|^2+\tau_n^2|\hat{z_n}|^2)}
{\Big[\ds\int_{\R^N}(\mu_1|u_n|^{2p}+2\beta\tau_n^p|u_n|^p|z_n|^p+\mu_2\tau_n^{2p}|z_n|^{2p})\Big]^{^\frac{1}{p}}}\\
&\geq\frac{(1+\tau_n^2)\min\Big\{\ds\int_{\R^N}(|\xi|^{2s}+1)|\hat{u_n}|^2,\ds\int_{\R^N}(|\xi|^{2s}+1)|\hat{z_n}|^2\Big\}}
{\Big(\mu_1+2\beta\tau_n^p+\mu_2\tau_n^{2p}\Big)^\frac{1}{p}\Big(\ds\int_{\R^N}|u_n|^{2p}\Big)^\frac{1}{p}}\\
&\geq f(\tau_n)S\geq f(\tau_{min})S,\\
\end{array}
\end{equation}
where we have used the facts that $\beta>0$ and  $\int_{\R^N}|u_n|^{2p}=\int_{\R^N}|z_n|^{2p}$.

Combining \eqref{2.4} and  \eqref{2.7}, we get
$$S_{\mu_1,\mu_2}=f(\tau_{min})S.$$

Since $w$ is the ground state of \eqref{1.3}, by direct computation, we find that $S_{\mu_1,\mu_2}$ can be obtained by $(w_{x_0}, \tau_{min}w_{x_0}).$ This completes the proof.
\end{proof}

\begin{remark}
In Lemma~\ref{lm3.1}, if $\tau_{\min}=0$, then the minimizer is semi-trivial.
\end{remark}

\begin{lemma}\label{lm3.2}
Assume that $\beta>0,\mu_1,\mu_2>0, 1<p<\frac{2^*_s}{2}$. We have

(1)\,\,if $p>2, 0<\beta\leq (p-1)\mu_2$, then $f(\tau)$ has a unique maximum point $\tau_0>1$ and a unique minimum point $0$;

\quad\quad if $p>2, \beta > (p-1)\mu_2$, then $f(\tau)$ has either a unique maximum point $\tau_0>1$ and a unique minimum point $0$,  or   two local minimum
points  $0, \tau_1$  and  two local maximum points $\tau_2, \tau_3$ satisfying  $0<\tau_2<\tau_1<\tau_3$;

(2)\,\, if $ p=2, \beta>\mu_1$, then $f(\tau)$ has a unique minimum point $\tau_0>1$ and a local maximum point $0$;

\quad\quad if $ p=2, 0<\beta<\mu_2$, then $f(\tau)$ has a unique maximum  point $\tau_0>1$ and a unique  minimum  point $0$;

(3)\,\,if $ 1<p<2, \beta \geq (p-1)\mu_2$, then $f(\tau)$ has a unique minimum   point $\tau_0 \in (0, 1)$ and  a local maximum   point $ 0$;

\quad\quad if $ 1<p<2, 0<\beta < (p-1)\mu_2$, then $f(\tau)$ has either  a unique minimum   point $\tau_0 \in (0, 1)$ and  a local maximum   point $ 0$,  or   two local maximum points
$0 ,\tau_1$ and   two local minimum points $\tau_2, \tau_3$ satisfying  $0<\tau_2<\tau_1<\tau_3$, which implies that $f(\tau)$ has
 a minimum point in $(0, +\infty)$.
\end{lemma}

The proof of Lemma~\ref{lm3.2} is very preliminary and  we postpone it to the Appendix.

\begin{proof}[ Proof of Theorem \ref{th1.6}:]
From Lemma \ref{lm3.2}, we know that $f(\tau)$ admits a minimum point $\tau_{min}>0$ under the assumptions of Theorem \ref{th1.6}. Hence  $(k_{min}w_{x_0}, k_{min}\tau_{min}w_{x_0})$  is a positive least energy solution of \eqref{1.1}, where $k_{min}>0$ satisfies $(\mu_1+\beta\tau_{min}^p)k_{min}^{2p-2}=1$. If we can prove that any positive least energy solution $(u_0, v_0)$ of \eqref{1.1} must be of the form $(k_{min}w_{x_0}, k_{min}\tau_{min}w_{x_0})$ and $\tau_{min}$ must be unique, then we complete the proof of Theorem \ref{th1.6}.

Firstly, we prove that
\begin{equation}\label{3.9}
\ds\int_{\R^N}|u_0|^{2p}=k_{min}^{2p}\ds\int_{\R^N}|w|^{2p}.
\end{equation}

To this end, we study the following equation with a parameter $\mu>0$
\begin{equation}\label{3.10}\left\{\begin{array}{ll}
(-\Delta)^su+u=\mu|u|^{2p-2}u+\beta|v|^p|u|^{p-2}u, x\in \R^N,\vspace{2mm}\\
(-\Delta)^sv+v=\mu_2|v|^{2p-2}v+\beta|u|^p|v|^{p-2}v, x\in \R^N,\vspace{2mm}\\
u,v \in H^s(\R^N).
\end{array}
\right.
\end{equation}

 Similarly, we denote
 $$f_1(\tau)=\frac{1+\tau^2}{(\mu+2\beta\tau^p+\mu_2\tau^{2p})^\frac{1}{p}},\,\,\,\, f_1(\tilde{\tau}_{min})=\min\limits_{\tau \geq 0}f_1(\tau).$$

 Since $0<\tau_{min}$ is a local minimum point of $f(\tau)$, then $f^\prime(\tau_{min})=0$ and  $f^{\prime\prime}(\tau_{min})>0$. From this, we  get $g(\tau_{min})=0$
  and  $g^{\prime}(\tau_{min})>0.$
 Let
 $$
 F(\mu, \tau)=\mu+\beta\tau^p-\beta\tau^{p-2}-\mu_2\tau^{2p-2}.
 $$
  Then
 \begin{eqnarray*}
 && F(\mu_1, \tau_{min})=g(\tau_{min})=0,\\
 &&\frac{\partial F(\mu, \tau)}{\partial \mu}=1,\\
 &&\frac{\partial F(\mu, \tau)}{\partial \tau}=\beta p\tau^{p-1}-\beta(p-2)\tau^{p-3}-\mu_2(2p-2)\tau^{2p-3},\\
 &&\frac{\partial F(\mu, \tau)}{\partial \tau}|_{_{_{(\mu, \tau)=(\mu_1, \tau_{min})}}}=g^{\prime}(\tau_{min})>0.
 \end{eqnarray*}
 By the Implicit Function Theorem, we can find $\varepsilon, \delta >0$ and two positive  functions $\tilde{k}_{min}(\mu), \tilde{\tau}_{min}(\mu)\in C^1((\mu_1-\varepsilon, \mu_1+\varepsilon),~(\tau_{min}-\delta, \tau_{min}+\delta))$ such that
 $F(\mu, \tilde{\tau}_{min})\equiv 0$ in $(\mu_1-\varepsilon, \mu_1+\varepsilon)$ and $(\mu+\beta\tilde{\tau}_{min}^p)\tilde{k}_{min}(\mu)^{2p-2}\equiv1$ in $(\mu_1-\varepsilon, \mu_1+\varepsilon)$. That is,
 $$\mu+\beta\tilde{\tau}_{min}^p-\beta\tilde{\tau}_{min}^{p-2}-\mu_2\tilde{\tau}_{min}^{2p-2}\equiv 0~\hbox{in}~\mu \in(\mu_1-\varepsilon, \mu_1+\varepsilon).$$
 Moreover, the least energy
 $B(\mu)=\tilde{k}_{min}^2(\mu)(1+ \tilde{\tau}^2_{min}(\mu))B_1 \in C^1((\mu_1-\varepsilon, \mu_1+\varepsilon),\R),$
 where $B_1=\frac{p-1}{2p}\int_{\R^N}|w|^{2p}$.

 By direct computation we can have
 \begin{equation}\label{3.10}
 B(\mu)=\inf\limits_{(u,v)\in(H^s(\R^N)\setminus\{0\})^2}\max\limits_{t>0}I_\mu(tu, tv)
 \end{equation}
 where
 $$
 I_\mu(u,v):=\frac{1}{2}\ds\int_{\R^N}(1+|\xi|^{2s})(|\hat{u}|^2+|\hat{v}|^2)-\frac{1}{2p}\ds\int_{\R^N}
 (\mu|u|^{2p}+2\beta|u|^p|v|^p+\mu_2|v|^{2p}).
 $$

 Define
 $$
 A:=\ds\int_{\R^N}(1+|\xi|^{2s})(|\hat{u}_0|^2+|\hat{v}_0|^2),\quad B:=\ds\int_{\R^N}
 (2\beta|u_0|^p|v_0|^p+\mu_2|v_0|^{2p}),\quad E:=\ds\int_{\R^N}|u_0|^{2p}.
 $$
 By direct computation, there exists a unique $t(\mu)>0$ such that
 $$\max\limits_{t>0}I_\mu(tu_0, tv_0)=I_\mu(t(\mu)u_0, t(\mu)v_0),$$
 where $t(\mu)=(\frac{A}{E\mu+B})^\frac{1}{2p-2}$.

 Let $H(\mu,t):=(E\mu+B)t^{2p-2}-A$. Then $H(\mu_1,1)=0, \frac{\partial H}{\partial t}(\mu_1, 1)>0$.
 By the Implicit Function Theorem, there exist $t(\mu), \varepsilon_1 \in(0, \varepsilon)$ such that
 $t(\mu) \in C^1((1-\varepsilon_1, 1+\varepsilon_1),\R)$ and
 $$t^\prime(\mu_1)=-\frac{E}{2(p-1)(E\mu_1+B)}.$$

 By Taylor expansion, we see that $t(\mu)=1+t^\prime(\mu_1)(\mu-\mu_1)+O((\mu-\mu_1)^2)$ and so
 $t^2(\mu)=1+2t^\prime(\mu_1)(\mu-\mu_1)+O((\mu-\mu_1)^2)$.
 By the fact that $(u_0, v_0)$ is a positive  least energy  solution of \eqref{1.1}, we have
 $$B(\mu_1)=\frac{p-1}{2p}A=\ds\frac{p-1}{2p}(E\mu_1+B).$$
Then using \eqref{3.10}, we get
\begin{equation}\label{3.11}\begin{array}{ll}
B(\mu)&\leq I_\mu(t(\mu)u_0, t(\mu)v_0)=\ds\frac{p-1}{2p}At(\mu)^2=B(\mu_1)t(\mu)^2\vspace{2mm}\\
&=B(\mu_1)-\ds\frac{E}{2p}(\mu-\mu_1)+O((\mu-\mu_1)^2).
\end{array}
\end{equation}
It follows that
$$\frac{B(\mu)-B(\mu_1)}{\mu-\mu_1} \geq -\frac{E}{2p}+O((\mu-\mu_1))$$
as $\mu \nearrow \mu_1$ and so $B^\prime(\mu_1)\geq-\frac{E}{2p}.$
Similarly, we have
$$\frac{B(\mu)-B(\mu_1)}{\mu-\mu_1} \leq -\frac{E}{2p}+O((\mu-\mu_1))$$
as $\mu \searrow \mu_1$, which means   $B^\prime(\mu_1)\leq-\frac{E}{2p}.$

Therefore,
$$
B^\prime(\mu_1)=-\frac{E}{2p}=-\frac{1}{2p}\ds\int_{\R^N}|u_0|^{2p}.
$$
Moreover, since $(k_{min}w, k_{min}\tau_{min}w)$ is  a positive  least energy solution of \eqref{1.1}, we have
$B^\prime(\mu_1)=-\frac{k_{min}^{2p}}{2p}\int_{\R^N}|w|^{2p}.$ So we have
\begin{equation}\label{3.12}
\ds\int_{\R^N}|u_0|^{2p}=k_{min}^{2p}\ds\int_{\R^N}|w|^{2p}.
\end{equation}
Moreover, we claim that $\tau_{min}$ is unique. In fact, suppose, to the contrary, that there exist two minimum points  $\tau^1_{min}\neq \tau^2_{min}$.
We have $k^1_{min}\neq k^2_{min}$. From the above proof, we deduce
$$
-\frac{(k^2_{min})^{2p}}{2p}\ds\int_{\R^N}|w|^{2p}=B^\prime(\mu_1)=-\frac{(k^1_{min})^{2p}}{2p}\ds\int_{\R^N}|w|^{2p},
$$
which contradicts to the fact that $k^1_{min}\neq k^2_{min}$. So $\tau_{min}$ must be unique.

With the   similar argument, we can show that
\begin{equation}\label{3.13}
\ds\int_{\R^N}|v_0|^{2p}=k^{2p}_{min}\tau_{min}^{2p}\ds\int_{\R^N}|w|^{2p},~~
\ds\int_{\R^N}|u_0|^p|v_0|^p=k^{2p}_{min}\tau_{min}^p\ds\int_{\R^N}|w|^{2p}.
\end{equation}

 Since $(k_{min}w, k_{min}\tau_{min}w)$ is a positive  least energy solution of \eqref{1.1}, we have
 \begin{equation}\label{3.14}
 (\mu_1+\beta\tau_{min}^p)k_{min}^{2p-2}=1=(\mu_2\tau_{min}^{2p-2}+\beta\tau_{min}^{p-2})k_{min}^{2p-2}.
 \end{equation}

Set $(u_1, v_1):=(\frac{u_0}{k_{min}}, \frac{v_0}{k_{min}\tau_{min}})$. It follows from  \eqref{3.12} and \eqref{3.14} that
$$\ds\int_{\R^N}(1+|\xi|^{2s})|\hat{u}_1|^2=\ds\int_{\R^N}|u_1|^{2p}.$$
Similarly, we find
$$\ds\int_{\R^N}(1+|\xi|^{2s})|\hat{v}_1|^2=\ds\int_{\R^N}|v_1|^{2p}.$$
Since $w$ is the ground state of \eqref{1.3},  we have
$$\ds\int_{\R^N}(1+|\xi|^{2s})|\hat{u}_1|^2\geq\ds\int_{\R^N}(1+|\xi|^{2s})|\hat{w}|^2,$$
and
$$\ds\int_{\R^N}(1+|\xi|^{2s})|\hat{v}_1|^2\geq\ds\int_{\R^N}(1+|\xi|^{2s})|\hat{w}|^2.$$

Noticing that $(u_0, v_0)$ and $(k_{min}w, k_{min}\tau_{min}w)$ are the least energy solutions of \eqref{1.1}, we obtain
\begin{eqnarray*}
&&\ds\frac{p-1}{2p}k^2_{min}(1+\tau^2_{min})\ds\int_{\R^N}(1+|\xi|^{2s})|\hat{w}|^2\vspace{2mm}\\
&=&\ds\frac{p-1}{2p}\ds\int_{\R^N}(1+|\xi|^{2s})(|\hat{u}_0|^2+|\hat{v}_0|^2)\\
&=&\ds\frac{p-1}{2p}\ds\int_{\R^N}(1+|\xi|^{2s})(k^2_{min}|\hat{u}_1|^2+k^2_{min}\tau^2_{min}|\hat{v}_1|^2)\vspace{2mm}\\
&\geq&\ds\frac{p-1}{2p}k^2_{min}(1+\tau^2_{min})\ds\int_{\R^N}(1+|\xi|^{2s})|\hat{w}|^2,
\end{eqnarray*}
which implies that
$$\ds\int_{\R^N}(1+|\xi|^{2s})|\hat{u}_1|^2 =\ds\int_{\R^N}(1+|\xi|^{2s})|\hat{w}|^2,$$
and
$$\ds\int_{\R^N}(1+|\xi|^{2s})|\hat{v}_1|^2=\ds\int_{\R^N}(1+|\xi|^{2s})|\hat{w}|^2.$$
So $u_1$ and $v_1$ are both positive least energy solutions of \eqref{1.3}.

By H\"{o}lder inequality, \eqref{3.12} and \eqref{3.13}, we see
\begin{equation}\label{H}
\begin{array}{ll}
\ds\int_{\R^N}w^{2p}&=\ds\int_{\R^N}|u_1|^{p}|v_1|^p\le \ds\frac12\ds\int_{\R^N}|u_1|^{2p}+\ds\frac12\ds\int_{\R^N}|v_1|^{2p}\vspace{2mm}\\
&=\ds\frac12\ds\int_{\R^N}w^{2p}+\ds\frac12\ds\int_{\R^N}w^{2p}\vspace{2mm}\\
&=\ds\int_{\R^N}w^{2p}.
\end{array}
\end{equation}

Hence the inequality in \eqref{H} is in fact an equality, which implies  $u_1=v_1$.

\end{proof}

\begin{remark}\label{re3.4}
From Lemma~\ref{lm3.2} and the proof of Theorem~\ref{th1.6}, we see

(1)\,\, Both  non-zero minimizers and non-zero maximizers of $f(\tau)$ correspond  positive proportional vector solutions of problem~\eqref{1.1}.

(2)\,\, In the cases $p>2, 0<\beta\leq (p-1)\mu_2$ and  $ p=2, 0<\beta<\mu_2$, problem~\eqref{1.1} admits no positive least energy solutions.

(3)\,\, When $p>2$ and $ \beta > (p-1)\mu_2$, $f(\tau)$ may have two local minimum points $0,\tau_1$. If we can prove $f(\tau_1) \leq f(0)$, then \eqref{1.1} has a unique positive least energy solution.

\end{remark}

\section{Appendix}\label{apx}
\begin{proof} [\bf Proof of Lemma \ref{lm3.2}:]
Set
$$H(\tau):=\mu_1+2\beta\tau^p+\mu_2\tau^{2p},$$
 and
 $$h(\tau):=\beta p\tau^2-\mu_22(p-1)\tau^p-\beta(p-2).$$
Since $\mu_1,\mu_2>0$, for $\beta>0$,  we have
 \begin{equation}\label{2.8}
 H(\tau)>0~\hbox{ for~ all}~ \tau \in (0,+\infty).
 \end{equation}
 By direct computation, we have
 \begin{equation}\label{2.9}
 f^\prime(\tau)=\frac{2\tau g(\tau)}{H^{\frac{1}{p}+1}(\tau)},
 \end{equation}

 \begin{equation}\label{2.10}
 g^\prime(\tau)=\tau^{p-3}h(\tau)
 \end{equation}
and
\begin{equation}\label{2.11}
h^\prime(\tau)=2p\tau(\beta-\mu_2(p-1)\tau^{p-2}).
\end{equation}

Combining Lemma \ref{lm2.3}, \eqref{2.8}--\eqref{2.11} and the fact that $\lim\limits_{\tau \to +\infty}f(\tau)=\mu_2^{-\frac{1}{p}}>\mu_1^{-\frac{1}{p}}=f(0)$,
we proceed the following discussion.

Case I: $\beta>0, p>2$

From \eqref{2.11}, we see that $h^\prime (\tau)=0$ has a unique positive solution $ \tau_2=(\frac{\beta}{\mu_2(p-1)})^\frac{1}{p-2}$,  and  $h^{\prime\prime} (\tau)=0$ has a unique positive solution $ \tau_1=(\frac{\beta}{\mu_2(p-1)^2})^\frac{1}{p-2}$, $\tau_1<\tau_2$. We have the following two subcases.

\quad\quad $(I_1): 0<\beta \leq (p-1)\mu_2$

In the case, $h(\tau_2) \leq 0$. We find
\begin{table}[!hbp]
\begin{tabular}{|c|c|c|c|c|c|c|c|}
\hline
 & $(0,\tau_1)$ &$\tau_1$ & $(\tau_1, \tau_2)$  & $\tau_2$& $(\tau_2, +\infty)$\\
 $h^{\prime\prime} (\tau)$ & $>0$ &  $=0$  &  $<0$ & $<0$& $<0$ \\
\hline
$h^\prime (\tau)$ & $>0$ &  $>0$  &  $>0$ & $=0$& $<0$ \\
\hline
$h(\tau)$ & $<0$ &  $<0$  &  $<0$ &  $\leq0$& $<0$\\
\hline
$g^\prime(\tau)$ & $<0$ &  $<0$  &  $<0$ &  $\leq0$& $<0$\\

\hline
\end{tabular}
\caption{$p>2$}
\end{table}

Considering that $g(\tau)=0$ has a positive solution $\tau_0$, we get the following table
\\
\\
 \begin{table}[!hbp]
\begin{tabular}{|c|c|c|c|c|}

\hline
 & $(0,\tau_0)$ & $\tau_0$ & $(\tau_0, +\infty)$  \\
\hline
$g(\tau)$ & $>0$ & $=0$ &  $<0$   \\
\hline
$f^\prime(\tau)$ & $>0$ & $=0$ &  $<0$  \\
\hline
\end{tabular}
\caption{$p>2$}
\end{table}

So $f(\tau)$ has a unique maximum point $\tau_0>1$ and a unique minimum point $0$.\\

\quad\quad $(I_2): \beta > (p-1)\mu_2$

In this case,  $h(\tau_2)>0$. Combining  the facts that $\lim\limits_{\tau \to 0^+}h(\tau)<0~\hbox{and}~
 \lim\limits_{\tau +\infty}h(\tau)=-\infty$ with the following table,  then $ h(\tau)=0$ has only two solution
$\tau_3, \tau_4$.
\\
\\
\\
\\
\begin{table}[!hbp]
\begin{tabular}{|c|c|c|c|c|c|c|c|}

\hline
 & $(0,\tau_1)$ &$\tau_1$ & $(\tau_1, \tau_2)$  & $\tau_2$& $(\tau_2, +\infty)$\\
 $h^{\prime\prime} (\tau)$ & $>0$ &  $=0$  &  $<0$ & $<0$& $<0$ \\

 \hline
$h^\prime (\tau)$ & $>0$ &  $>0$  &  $>0$ & $=0$& $<0$ \\
\hline
\end{tabular}
\caption{$p>2$}
\end{table}

So we have

\begin{table}[!hbp]
\begin{tabular}{|c|c|c|c|c|c|c|c|}

\hline
 & $(0,\tau_3)$ &$\tau_3$ & $(\tau_3, \tau_4)$  & $\tau_4$& $(\tau_4, +\infty)$\\
 $h (\tau)$ & $<0$ &  $=0$  &  $>0$ & $=0$& $<0$ \\
\hline

$g^\prime (\tau)$ & $<0$ &  $=0$  &  $>0$ & $=0$& $<0$ \\
\hline
\end{tabular}
\caption{$p>2$}
\end{table}
From the above table we can see that $g(\tau)=0$ has at most three solutions. If $g(\tau)=0$ has one solution or two solutions, then we can find a solution $\tilde{\tau}_1>0$ of $f^\prime (\tau)=0$ such that
$f(\tau)$  increases strictly in $(0, \tilde{\tau}_1)\setminus\{\tilde{\tau}_2\}$ and $f(\tau)$  decreases strictly  in $(\tilde{\tau}_1, +\infty)\setminus\{\tilde{\tau}_2\}$,
where $\tilde{\tau}_2$ is the other root of $f^\prime (\tau)=0$ if exists.Therefore,  we can see  that $f(\tau)$ has a unique maximum point $\tau_0>1$ and a unique minimum point $0$. Now  we   study the case that $g(\tau)=0$ has  three solutions $\tau_5,\tau_6,\tau_7$.
From the following  table we can see that $\tau_5, \tau_7$ are the local maximum points of $f(\tau)$ and $0, \tau_6$ are the local minimum points of $f(\tau)$.
\\
\\
\\
\\
\\

\begin{table}[!hbp]
\begin{tabular}{|c|c|c|c|c|c|c|c|c|c|}

\hline
 & $(0,\tau_5)$ &$\tau_5$ & $(\tau_5, \tau_6)$  & $\tau_6$& $(\tau_6, \tau_7)$& $\tau_7$ & $(\tau_7, +\infty)$\\
 $g (\tau)$&$>0$&$=0$ & $<0$ &  $=0$  &  $>0$ & $=0$& $<0$ \\
\hline

$f^\prime (\tau)$ &$>0$&$=0 $& $<0$ &  $=0$  &  $>0$ & $=0$& $<0$ \\
\hline
\end{tabular}
\caption{$p>2$}
\end{table}

Case II: $\beta>0, p=2$

We see $g(\tau)=(\mu_1-\beta)+(\beta-\mu_2)\tau^2$.

\quad  $(II_1):$~~ $0<\beta< \mu_2$

 We have

 \begin{table}[!hbp]
\begin{tabular}{|c|c|c|c|c|}
\hline
 & $(0,\tau_0)$ & $\tau_0=\sqrt{\frac{\mu_1-\beta}{\mu_2-\beta}}$ & $(\tau_0, +\infty)$  \\
\hline
$g(\tau)$ & $>0$ & $=0$ &  $<0$   \\
\hline
$f^\prime(\tau)$ & $>0$ & $=0$ &  $<0$  \\
\hline
\end{tabular}
\caption{$0<\beta< \mu_2,p=2$}
\end{table}
From this, we can see that  $f(\tau)$ has a unique maximum point $\tau_0>1$ and a unique minimum point $0$.

\quad $(II_2):$~~~ $\beta> \mu_1$

If $\beta> \mu_1$, we can obtain

 \begin{table}[!hbp]
\begin{tabular}{|c|c|c|c|c|}
\hline
 & $(0,\tau_0)$ & $\tau_0=\sqrt{\frac{\mu_1-\beta}{\mu_2-\beta}}$ & $(\tau_0, +\infty)$  \\
\hline
$g(\tau)$ & $<0$ & $=0$ &  $>0$   \\
\hline
$f^\prime(\tau)$ & $<0$ & $=0$ &  $>0$  \\
\hline
\end{tabular}
\caption{$\beta> \mu_1, p=2$}
\end{table}
So we also can get that $f(\tau)$ has a unique minimum point $\tau_0>1$ and a local  maximum point $0$
in $[0, +\infty)$.

Case III: $\beta>0, 1<p<2$

By direct computation,
we find that  $h^{\prime}(\tau)=0$ has a unique positive solution $ \tau_2=(\frac{\beta}{\mu_2(p-1)})^\frac{1}{p-2}$,
$h^{\prime}(\tau)<0$ in $(0, \tau_2)$ and $h^{\prime}(\tau)>0$ in $(\tau_2, +\infty)$.

\quad $(III_1):$ $\beta \geq (p-1)\mu_2$

Direct computation yields $h(\tau) >0$ for $\tau \in (0,+\infty)\setminus\{\tau_2\}$ and $g^\prime (\tau) >0$ for $\tau \in (0,+\infty)\setminus\{\tau_2\}.$
So $g(\tau)$ increases in $(0,+\infty)$. Due to the fact that $\lim\limits_{\tau \to 0^+}g(\tau)=-\infty, g(1)>0,$ we deduce that $g(\tau)=0$ has a unique positive solution $\tau_0<1$,
$g(\tau)<0 $ in $(0, \tau_0)$ and $g(\tau)>0 $ in $(\tau_0, +\infty)$.  Therefore we have proved that $f(\tau)$ has a unique minimum point $\tau_0<1$ and a local maximum point $0$.

\quad $(III_2):$ $0<\beta<(p-1)\mu_2$

We see
$h(\tau_2)<0$. So $h(\tau)=0$ has two roots $\tau_3, \tau_4$. We have the following table.
\\
\\
\\
\begin{table}[!hbp]
\begin{tabular}{|c|c|c|c|c|c|c|c|}
\hline
 & $(0,\tau_3)$ &$\tau_3$ & $(\tau_3, \tau_4)$  & $\tau_4$& $(\tau_4, +\infty)$\\
 $h (\tau)$ & $>0$ &  $=0$  &  $<0$ & $=0$& $>0$ \\
\hline

$g^\prime (\tau)$ & $>0$ &  $=0$  &  $<0$ & $=0$& $>0$ \\
\hline
\end{tabular}
\caption{$1<p<2$}
\end{table}

From the above table, we can see that  $g(\tau)=0$ and hence $f^\prime (\tau)=0$  have at most three solutions.  If $f^\prime (\tau)=0$ has at most two  solutions, then we can find a solution $\tilde{\tau}_1>0$ of $f^\prime (\tau)=0$ such that
$f(\tau)$  decreases strictly in $(0, \tilde{\tau}_1)\setminus\{\tilde{\tau}_2\}$ and $f(\tau)$  increases strictly  in $(\tilde{\tau}_1, +\infty)\setminus\{\tilde{\tau}_2\}$,
where $\tilde{\tau}_2$ is the other root of $f^\prime (\tau)=0$ if exists.
 So $f(\tau)$ has a unique minimum point $\tau_0<1$ and a local maximum point $0$.
 Now we study the case that $f^\prime (\tau)=0$ has three solutions $\tau_5, \tau_6, \tau_7$.
\\
\begin{table}[!hbp]
\begin{tabular}{|c|c|c|c|c|c|c|c|c|c|}

\hline
 & $(0,\tau_5)$ &$\tau_5$ & $(\tau_5, \tau_6)$  & $\tau_6$& $(\tau_6, \tau_7)$& $\tau_7$ & $(\tau_7, +\infty)$\\
 $g (\tau)$&$<0$&$=0$ & $>0$ &  $=0$  &  $<0$ & $=0$& $>0$ \\
\hline

$f^\prime (\tau)$ &$<0$&$=0 $& $>0$ &  $=0$  &  $<0$ & $=0$& $>0$ \\
\hline
\end{tabular}
\caption{$p>2$}
\end{table}

From the above table,  we see that $\tau_5, \tau_7$ are the local minimum points of $f(\tau)$, and
$0, \tau_6$ are the local maximum points of $f(\tau)$ and $f(\tau_5)<f(0)$. Therefore  $f(\tau)$ has
a minimum point and $\min\{f(\tau): \tau \geq 0\}<f(0)$.

\end{proof}


\begin{thebibliography}{999}
\bibitem{a} C. Alves, Local mountain pass for a class of elliptic system, \it J. Math. Anal. Appl., \bf 335\rm(2007), 135-150.

\bibitem{at} C. Amick and  J. Toland,
Uniqueness and related analytic properties for the Benjamin-Ono equation-a nonlinear Neumann problem in the plane, \it
Acta Math., \bf 167\rm(1991), no. 1-2, 107-126.

\bibitem{bdw} T.Bartsch,  N.Dancer and Z.Wang,
A Liouville theorem, a-priori bounds, and bifurcating branches of positive solutions for a nonlinear elliptic system, \it
Calc. Var. Partial Differential Equations, \bf 37\rm(2010), no. 3-4, 345-361.

\bibitem{bl} J.Bona and  Y.Li, Decay and analyticity of solitary waves, \it J. Math. Pures Appl., \bf
(9)76\rm(1997), no. 5, 377-430.

\bibitem{bs} A.de Bouard and  J.Saut, Symmetries and decay of the generalized Kadomtsev-Petviashvili solitary waves, \it SIAM J. Math. Anal., \bf 28\rm(1997), no. 5, 1064-1085.


\bibitem{c} X. Chang, Ground state solutions of asymptotically linear fractional Schr\"{o}dinger equations, \it J. Math. Phys., \bf 54\rm(2013), no. 6, 061504, 10pp.

\bibitem{clw} X. Chen, T.  Lin and  J. Wei, Blow up and solitary wave solutions with ring profiles of two-component nonlinear Schr\"{o}dinger systems,
\it Phys. D.,  \bf 239\rm(2010), 613-626.

 \bibitem{c1} C. Coffman, Uniqueness of the ground state solution for $\Delta u-u+u^3=0$ and a variational
characterization of other solutions, \it Arch. Rational Mech. Anal.,  \bf46\rm(1972),  81-95.

\bibitem{cwy} W. Chen, J. Wei and  S. Yan, Infinitely many solutions for the Schr\"{o}dinger equations in $\R^N$ with critical growth,
\it J. Differential Equations, \bf 252\rm(2012), \rm 2425-2447.
\bibitem{cz}Z. Chen and W. Zou,
Positive least energy solutions and phase separation for coupled Schr\"{o}dinger equations with critical exponent.
Arch. Ration. Mech. Anal. 205 (2012), no. 2, 515-551.

\bibitem{dds} J. D\'{a}vila, M. del Pino and  Y. Sire, Nondegeneracy of the bubble in the critical case for nonlocal equations, \it Proc. Amer. Math. Soc., \bf 141\rm(2013), 3865-3870.

 \bibitem{den} E. Dancer, On the influence of domain shape on the existence of large solutions of some superlinear problem,
\it Math. Ann., \bf 285\rm(1989), \rm 647-669.

\bibitem{dp} S. Dipierro and  A. Pinamonti,  A geometric inequality and a symmetry result for elliptic systems involving the fractional Laplacian, \it J. Differential Equations, {\bf 255}\rm(2013), no. 1, 85-119.

\bibitem{dpv} S. Dipierro, E. Palatucci and  E. Valdinoci,
Existence and symmetry results for a Schr\"{o}dinger type problem involving the fractional Laplacian, \it
Matematiche (Catania), \bf 68\rm(2013), no. 1, 201-216.

\bibitem{dy} E. Dancer and  S.Yan, Multibump solutions for an elliptic problem in expanding domains,
\it Comm. Partial Differ. Equ., \bf 27\rm(2002), \rm 23-55.

\bibitem{fl} R. Frank and  E.Lenzmann, Uniqueness of non-linear ground states
for fractional Laplacians in $\R$, \it Acta Math., \bf 210\rm(2013), 261-318.

\bibitem{fls} R. Frank, E.Lenzmann and  L.Silvester, uniqueness of radial solutions for the fractional laplacian,
arXiv:1302.2652v1 [math.AP] 11 Feb 2013.

\bibitem{fqt} P. Felmer, A.Quaas and J.Tan,  Positive solutions of the nonlinear Schr\"{o}dinger equation with the fractional Laplacian, \it Proc. Roy. Soc. Edinburgh Sect. A, \bf 142\rm(2012), no. 6, 1237-1262.

\bibitem{kmr} C. Kenig, Y. Martel and  L. Robbiano, Local well-posedness and blow-up in the energy space
for a class of $L^2$critical dispersion generalized Benjamin-Ono equations, \it Ann. Inst. H. Poincar\'{e}
Anal. Non Lin\'{e}aire, \bf 28\rm(2011), no. 6, 853-887.

\bibitem{k} M. Kwong, Uniqueness of positive solutions of $\Delta u-u + u^p = 0 $ in $\R^N$,  \it Arch. Rational Mech.
Anal., \bf 105\rm(1989), 243-266.

\bibitem{l1} N. Laskin, Fractional quantum mechanics and L\'{e}vy path integrals,  \it Phys. Lett. A, \bf 268\rm(2000), 298-305.

\bibitem{l2} N. Laskin, Fractional Schr\"{o}dinger equation, \it Phys. Rev. E, \bf 66\rm(2002), 31-35.

\bibitem{lw2} T. Lin and  J. Wei, Spikes in two-component systems of nonlinear Schr\"{o}dinger equations with trapping potentials,
\it J. Differential Equations,  \bf 229 \rm(2006),  538-569.

\bibitem{m} M. Maris, On the existence, regularity and decay of solitary waves to a generalized
Benjamin-Ono equation, \it Nonlinear Anal., \bf 51 \rm(2002), no. 6, 1073-1085.

\bibitem{mc} K. McLeod, Uniqueness of positive radial solutions of $\Delta u+f(u) = 0$ in $\R^N$, II,  \it Trans. Amer. Math.
Soc.,  \bf 339\rm(1993), 495-505.

\bibitem{mr} F. Merle and  P. Raphael, The blow-up dynamic and upper bound on the blow-up rate for critical
nonlinear Schr\"{o}dinger equation,  \it Ann. Math., {\bf 161}\rm(2005), 157-222.

\bibitem{pw} S. Peng and  Z. Wang, Segregated and synchronized vector solutions for nonlinear Schr\"{o}dinger systems, \it  Arch. Ration. Mech. Anal.,
 \bf 208\rm(2013), no. 1, 305-339.

\bibitem{p} A. Pomponio, Coupled nonlinear Schr\"{o}dinger systems with potentials, \it J. Differential Equations, \bf 227
\rm(2006), 258-281.

\bibitem{s} S. Secchi, Ground state solutions for nonlinear fractional Schr\"{o}dinger equations in $\R^N$, \it J. Math. Phys., \bf  54\rm(2013), no. 3, 031501, 17pp.

\bibitem{w} M. Weinstein, Solitary waves of nonlinear dispersive evolution equations with critical
power nonlinearities, J. Differential Equations 69 (1987), no. 2, 192-203.

\bibitem{wy} J. Wei and  S. Yan, Infinitely many positive solutions for the nonlinear Schr\"{o}dinger equations in $\R^N$,
     \it Calc. Var. Partial Differential Equations, \bf 37\rm(2010), \rm 423-439.

\end{thebibliography}
\end{document}